\documentclass[11pt]{amsart}
\usepackage{xcolor}
\usepackage[margin=1in]{geometry}
\usepackage[colorlinks=true]{hyperref}
\hypersetup{urlcolor=blue, citecolor=red}
 \numberwithin{equation}{section}

\newcommand{\ra}{\rightarrow}
\newcommand{\ve}{\varepsilon}

\newcommand{\pt}{\partial_t}

\newcommand{\RN}{\mathbb{R}^N}
	
	\newcommand{\br}{B_{r}(x_0)}

	\newcommand{\mdiv}{\textup{div}}

\newcommand*{\avint}{\mathop{\ooalign{$\int$\cr$-$}}}

\newcommand{\io}{\int_{\Omega}}
\newcommand{\iot}{\int_{\Omega_{\tau}}}

\newcommand{\air}{\avint_{B_r(x_0)}  }

\newcommand{\mnp}{(\m\cdot\nabla p) }
\newcommand{\nnp}{(\n\cdot\nabla p) }

\newcommand{\ott}{\Omega_\tau }

\newcommand{\f}{{\bf g}}

\newcommand{\po}{\partial\Omega}

  \newcommand{\w}{\mathbf{w}}
 \newcommand{\g}{\mathbf{g}}
  \newcommand{\m}{\mathbf{m}}
\newcommand{\n}{\mathbf{n}}  
\newcommand{\ioT}{\int_{\Omega_{T}}}

\newcommand{\ot}{\Omega_T }
\newtheorem{theorem}{Theorem}[section]

\newtheorem{lemma}[theorem]{Lemma}

\theoremstyle{definition}
\newtheorem{definition}[theorem]{Definition}

\newcommand{\ep}{\varepsilon}

\title[a biological network formation model
] 
      {Blow-up time of strong solutions to a biological network formation model in high space dimensions}

\author[Xiangsheng Xu]{}

\subjclass{Primary: 35B44, 35B65, 35D35, 35Q92, 35A01.}
 \keywords{ biological network formulation, blow-up time,  existence.}
 \email{xxu@math.msstate.edu}

\begin{document}
\maketitle

\centerline{\scshape Xiangsheng Xu}
\medskip
{\footnotesize
 \centerline{Department of Mathematics \& Statistics}
   \centerline{Mississippi State University}
   \centerline{ Mississippi State, MS 39762, USA}
} 

\bigskip

\begin{abstract}
We investigate the possible blow-up of strong solutions to a biological network formation model originally introduced
by D. Cai and D. Hu \cite{HC}. The model is represented by an initial boundary value problem for an elliptic-parabolic system with cubic non linearity. We obtain an algebraic equation for the possible blow-up time of strong solutions. The equation yields information on how various given data may contribute to the blow-up of solutions. As a by-product of our development, we establish a $W^{1,q}$ estimate for solutions to an elliptic equation which shows the explicit dependence of the upper bound on the elliptic coefficients.

\end{abstract}

\section{Introduction}
Let $\Omega$ be a bounded domain in $\mathbb{R}^N$ with $C^{1,1}$ boundary $\po$ and $T$ a positive number. Set $\ot=\Omega\times(0,T)$. We study the possible blow-up of strong solutions to the system
\begin{align}
-\mbox{{div}}\left[(I+\m\otimes \m)\nabla p\right]&=S(x)\ \ \ \mbox{in $\ot$,}\label{e1}\\
\partial_t\m-D^2\Delta \m+|\m|^{2(\gamma-1)}\m&=E^2\mnp\nabla p\ \ \ \mbox{in $\ot$,}\label{e2}
\end{align}
coupled with the initial boundary conditions
\begin{align}
p(x,t)=0, &\ \  \m(x,t)=0, \ \ \ (x,t)\in \Sigma_T\equiv\partial\Omega\times(0,T),\label{e3}\\
\m(x,0)&=\m_0(x),\ \ \ \ x\in\Omega\label{e4}
\end{align} 
for given functions $S(x),\m_0(x)$ and physical parameters $D, E, \gamma$ with properties:
\begin{enumerate}
	\item[(H1)] $\m_0(x)\in\left(W_0^{1,\infty}(\Omega)\right)^N$, {\color{red}{$N\geq 3$}}, $S(x)\in L^{\frac{4qN}{N+4q}}(\Omega)$ for some $ q>1+\frac{N}{2}$; and
	\item[(H2)] $D, E\in (0, \infty), \gamma\in (\frac{1}{2}, \infty)$.
\end{enumerate}
This system was originally derived in (\cite{H}, \cite{HC}) as the formal gradient flow of the continuous version of a cost functional describing formation of biological transportation networks on discrete graphs. In this context, the scalar function $p=p(x,t)$ is the pressure due to Darcy's law, while the vector-valued function $\m=\m(x,t)$ is the conductance vector. The function $S(x)$ is the time-independent source term. More detailed information on the biological relevance of the problem can be found in \cite{AAFM,ABHMS,HMPS}.

We are interested in the mathematical analysis of the problem. 
		A pair $(\m, p)$ is said to be a weak solution to \eqref{e1}-\eqref{e4} if:
\begin{enumerate}
	\item[(D1)] $\m\in L^\infty\left(0,T; \left(W^{1,2}_0(\Omega)\cap L^{2\gamma}(\Omega)\right)^N\right),\ \partial_t\m\in L^2\left(0,T; \left(L^2(\Omega)\right)^N\right),\  p\in L^\infty(0,T; W^{1,2}_0(\Omega)), \ \mnp \in L^\infty(0,T;  L^{2}(\Omega))$;
	\item[(D2)] $m(x,0)=m_0$ in $C\left([0,T]; \left(L^2(\Omega)\right)^N\right)$;
	\item[(D3)] Equations \eqref{e1} and \eqref{e2} are satisfied in the sense of distributions.
\end{enumerate}
The existence of a weak solution was first established in \cite{HMP}. It was based upon the following a priori estimates
\begin{eqnarray}
	\lefteqn{\frac{1}{2}\int_{\Omega}|\m(x,\tau)|^2dx+D^2\int_{\Omega_\tau}|\nabla \m|^2dxdt+E^2\int_{\Omega_\tau} \mnp^2dxdt}\nonumber\\
	&&+\int_{\Omega_\tau}|\m|^{2\gamma}dxdt+2E^2	\int_{\Omega_\tau} |\nabla p|^2dxd\tau\nonumber\\
	&=&\frac{1}{2}\int_{\Omega}|\m_0|^2dx+2E^2\int_{\Omega_\tau} S(x)pdxdt,\label{me1}\\
	\lefteqn{\int_{\Omega_\tau}|\partial_t\m|^2dxdt+\frac{D^2}{2}\int_{\Omega}|\nabla \m(x,\tau)|^2dx+\frac{E^2}{2}\int_{\Omega}\mnp^2 dx}\nonumber\\
	&&+\frac{E^2}{2}\int_{\Omega}|\nabla p|^2dx+\frac{1}{2\gamma}\int_{\Omega}|\m|^{2\gamma}dx\nonumber\\
	&=&\frac{D^2}{2}\int_{\Omega}|\nabla \m_0|^2dx+\frac{E^2}{2}\int_{\Omega}(\m_0\cdot \nabla p_0)^2dx+\frac{1}{2\gamma}\int_{\Omega}|\m_0|^{2\gamma}dx\nonumber\\
	&&+\frac{E^2}{2}\int_{\Omega}|\nabla p_0|^2dx,\label{f5}
\end{eqnarray}
where $\tau\in (0, T], \Omega_\tau=\Omega\times(0,\tau)$, and $p_0$ is the solution of the boundary value problem
\begin{eqnarray}
	-\mbox{div}[(I+\m_0\otimes \m_0)\nabla p_0] &=& S(x),\label{pzo}
	\ \ \ \mbox{in $\Omega$,}\\
	p_0&=& 0\ \ \ \mbox{on $\partial\Omega$.}\label{pzt}
\end{eqnarray}
 Partial regularity of weak solutions was addressed in \cite{LX,X3} for $N\leq 3$. If the space dimension is $2$, a classical solution was obtained in \cite{X5} for the stationary case and in \cite{X1} for the time-dependent case. Finite time extinction or break-down of solutions in the spatially one-dimensional setting for certain ranges of the relaxation exponent $\gamma$ was carefully studied in \cite{HMPS}. Further modeling analysis and numerical results can be found in \cite{AAFM}. We also mention that the question of existence in the case where $\gamma=\frac{1}{2}$ is addressed in \cite{HMPS}. In this case the term $|\m|^{2(\gamma-1)}\m$ is not continuous at $\m=0$. 
However, the general regularity theory remains fundamentally incomplete. In particular, it is not known whether or not weak solutions develop singularities in high space dimensions $N\geq 3$ even though a blow-up criterion was obtained in \cite{L} when $\Omega=\mathbb{R}^3$. 

A strong solution is a weak solution with the additional property
\begin{enumerate}
	\item[(D4)] $\m$ is H\"{o}lder continuous in $\overline{\ot}$.
\end{enumerate} 
If (D4) holds, then equation \eqref{e1} becomes uniformly elliptic. In fact, much better is true. Under (D4) and the assumption
	\begin{enumerate}
	\item[\textup{(H3)}] $\partial\Omega$ is $C^{1,1}$,
\end{enumerate}
the result in (\cite{R}, p.82) becomes applicable. Upon invoking the result, we conclude that for each $s>1$ there is a positive number $c$ determined by $N, s, \Omega$, and the H\"{o}lder continuity of $\m$ such that
\begin{equation}\label{r82}
	\|\nabla p\|_{s,\Omega}\leq c\|S\|_{\frac{sN}{s+N},\Omega}.
\end{equation}
This, in turn, will  further improve the regularity of $(\m,p)$. In fact, one can infer from the results in \cite{CFL1, LSU} that the system \eqref{e1}-\eqref{e2}
is satisfied a.e on $\ot$. We shall not elaborate.

Our main result is:
\begin{theorem}\label{main} Let  (H1)-(H3) be satisfied. 
Then there is a positive number $T_{\textup{max}}$ such that problem \eqref{e1}-\eqref{e4} has a strong solution in $\ot$ for each $T<T_{\textup{max}}$. The
	number $T_{\textup{max}}$ is the unique solution of the equation
	\begin{equation}\label{tmax}
		\left(\sum_{i=1}^7T_{\textup{max}}^{a_i}\|S\|_{\frac{4qN}{N+4q},\Omega}^{b_i}\|\nabla\m_0\|_{\infty, \Omega}^{c_i}\right)^s\sum_{i=8}^{19}T_{\textup{max}}^{a_i}\|S\|_{\frac{4qN}{N+4q},\Omega}^{b_i}\|\nabla\m_0\|_{\infty, \Omega}^{c_i}=d_0,
	\end{equation}
where all the exponents $s>0, a_i>0, b_i\geq 0, c_i\geq 0$ for $i=1,\cdots, 19$ are determined by $N, q, \gamma$ only, while $d_0>0$ also depends on $\Omega$ and the physical parameters $D, E$ in the problem in addition to $N, q, \gamma$.
%
\end{theorem}

This theorem immediately asserts the local existence of a strong solution. Obviously, the smaller $\|S\|_{\frac{4qN}{N+4q},\Omega}$ and $\|\nabla\m_0\|_{\infty, \Omega}$ are, the longer the life span of such a strong solution is. It is a refinement of Theorem 1.8 in \cite{X3}. However, the theorem does not say if the blow-up does occur at $T_{\textup{max}}$. Thus, $T_{\textup{max}}$ only serves as a lower bound for the possible blow-up time.
	

To describe the mathematical difficulty involved in our problem, first notice the term $\mnp \nabla p$ in \eqref{e2} represents a cubic non-linearity, which is still not well-understood. Second, the elliptic coefficients in \eqref{e1}
satisfy
\begin{equation}\label{ellip}
	|\xi|^2\leq  (I+\m\otimes \m)\xi\cdot\xi\leq (1+|\m|^2)|\xi|^2\ \ \ \mbox{for all $\xi\in \mathbb{R}^N$.}	
\end{equation}
This allows the possibility that the coefficient matrix oscillate wildly among different eigen-directions unless $\m$ is bounded, which is not known a priori. However, existing results for degenerate and/or singular elliptic equations such as these in \cite{HKM} often require that
the largest eigenvalue $\lambda_l$ and the smallest eigenvalue $\lambda_s$ of the coefficient matrix satisfy 
$$\lambda_l\leq c\lambda_s.$$
Here and in what follows the letter $c$ denotes a generic positive number. Thus our problem does not fit into the classical framework.

To gain some insights into our problem, we take a look at the one-dimensional case
\begin{eqnarray}
	-[\left(1+m^2\right)p_x]_x&=& S(x)\ \ \mbox{in $(0,1)\times(0,\infty)$},\label{1dp1}\\
	m_t-D^2m_{xx}	+|m|^{2(\gamma-1)}m&=&E^2mp_x^2\ \ \mbox{in $(0,1)\times(0,\infty)$},\label{1dp2}\\
	p(0,t)&=&p(1,t)=0\ \ \ \mbox{on $(0, \infty)$,}\label{1dp3}\\
	m(0,t)&=&m(1,t)=0\ \ \ \mbox{on $(0, \infty)$,}\label{1dp4}\\
	m(x,0)&=& m_0(x)\ \ \ \mbox{on $(0,1)$}.\label{1dp5}
\end{eqnarray}
According to Rolle's theorem, for each $t\in (0, \infty)$ there is an $x^*(t)\in (0,1)$ such that
\begin{equation}
	p_x(x^*(t),t)=0.\nonumber
\end{equation}
For $x\in (0,1)$ we integrate \eqref{1dp1} over $(x^*(t), x)$ to derive
\begin{equation}\label{1pe1}
	p_x=-\frac{1}{1+m^2}\int_{x^*(t)}^{x}S(y)dy.
\end{equation}
Subsequently,
\begin{equation}
	\|p_x\|_\infty\leq \|S\|_1.\nonumber
\end{equation}
Substitute \eqref{1pe1} into \eqref{1dp2} to obtain
\begin{eqnarray}
	m_t-D^2m_{xx}	+|m|^{2(\gamma-1)}m&=&\frac{E^2m}{(1+m^2)^2}\left(\int_{x^*(t)}^{x}S(y)dy\right)^2.\nonumber
\end{eqnarray}
We multiply through the above equation by $ m_t$ and integrate to obtain
\begin{eqnarray}
	\lefteqn{\int_{0}^{1}m_t^2dx+\frac{D^2}{2}\frac{d}{dt}\int_{0}^{1}m_x^2dx+\frac{1}{2\gamma}\frac{d}{dt}\int_{0}^{1}m^{2\gamma}dx}\nonumber\\
	&=&E^2\int_{0}^{1}\frac{mm_t}{(1+m^2)^2}\left(\int_{x^*(t)}^{x}S(y)dy\right)^2dx\nonumber\\
	&\leq&E^2\|S\|_1^2\int_{0}^{1}\frac{m^{2}}{(1+m^2)^2}dx.\nonumber
\end{eqnarray}
This implies that blow-up in $m$ does not occur. The key here is \eqref{1pe1}, which is not available in high space dimensions.
To seek a substitute, we have developed the following theorem which seems to be of interest on its own right.
\begin{theorem}\label{w1q} Assume that (H3) holds and \begin{equation}
		\w\in \left(W^{1,\ell}(\Omega)\right)^N\ \ \mbox{for some $\ell>N$}.
	\end{equation}
	Let $p$ be the solution of the boundary value problem
	\begin{eqnarray}
		-\mdiv\left[(I+\w(x)\otimes\w(x))\nabla p\right]&=& S(x)\ \ \mbox{in $\Omega$,}\label{peqn1}\\
		p&=&0\ \ \mbox{on $\partial\Omega$,}\label{peqn2}
	\end{eqnarray}
	where $S$ is a given function in a suitable function space.
	For each 
	\begin{equation}\label{conq}
		q>\frac{N}{N-1}
	\end{equation}
	there is a positive number $c=(N,\Omega, q,\ell)$ such that
	\begin{eqnarray}
		\|\nabla p\|_{q,\Omega}&\leq& c\left(1+\|\w\|_{\left(W^{1,\ell}(\Omega)\right)^N}\right)^{s_1}\left(\|\nabla p\|_{1,\Omega}+\|S\|_{\frac{Nq}{N+q},\Omega}\right),\label{npb10}
	\end{eqnarray}	
	where $s_1=\frac{5N(2\ell-N+N\ell)(Nq-N)}{q(\ell-N)}$. 
\end{theorem}
The advantage of this theorem over the classical result \eqref{r82} is that it gives the explicit dependence of the upper bound on the coefficient matrix. This is very crucial to our applications. 
Condition \eqref{conq} is to essure that $\frac{Nq}{N+q}>1$.
When $N=2$, a version of \eqref{npb10} was obtained in \cite{X1} by deriving an equation for the term $(I+\w\otimes\w)\nabla p\cdot\nabla p$. Unfortunately, this technique only works for $N=2$. Our approach here is largely inspired by the papers \cite{CFL,CFL1,DI}.


The rest of the paper is organized as follows. A refinement of a classical uniform bound for solutions to a linear parabolic
equation is given in Section 2. The proof of Theorem \ref{w1q} is presented in Section 3.  Our main theorem is established in Section 4. 

Finally, we remark that unless stated otherwise our generic constant $c$ depends only on $N, q, \Omega$ and the three physical parameters $D, E, \gamma$ in the problem. In particular, it does not depend on $T, \m_0(x)$, or $S(x)$. The following two inequalities are frequently used without acknowledgment:
\begin{eqnarray}
	(a+b)^\alpha&\leq& a^\alpha+ b^\alpha\ \ \mbox{for $a\geq 0, b\geq 0, \alpha\in (0,1)$,}\\
	(a+b)^\alpha&\leq& 2^{\alpha-1}\left(a^\alpha+ b^\alpha\right)\ \ \mbox{for $a\geq 0, b\geq 0, \alpha\geq 1$.}
\end{eqnarray}
\section{Preliminary results} In this section, we collect a few relevant known results.
The first lemma contains some elementary inequalities whose proof can be found in (\cite{O}, p. 146-148). 
\begin{lemma}\label{plap}Let $x,y$ be any two vectors in $\RN$. Then:
	\begin{enumerate}
		\item[\textup{(i)}] For $\gamma\geq 1$,
		\begin{equation*}
		\left(\left(|x|^{2\gamma-2}x-|y|^{2\gamma-2}y\right)\cdot(x-y)\right)\geq \frac{1}{2^{2\gamma-1}}|x-y|^{2\gamma};
		\end{equation*}
		\item[\textup{(ii)}] For $\frac{1}{2}<\gamma\leq 1$,
		\begin{equation*}
		\left(|x|+|y|\right)^{2-2\gamma}\left(\left(|x|^{2\gamma-2}x-|y|^{2\gamma-2}y\right)\cdot(x-y)\right)\geq (2\gamma-1)|x-y|^2.
		\end{equation*}
	\end{enumerate}
\end{lemma}
The next lemma plays a central role in our main result. 
\begin{lemma}\label{prop2.2}
	Let $h(\tau)$ be a continuous non-negative function defined on $[0, T_0]$ for some $T_0>0$. Suppose that there exist three positive numbers $\ep, \delta, b $ such that
	\begin{equation}\label{f22}
		h(\tau)\leq \ep h^{1+\delta}(\tau)+b\ \ \mbox{for each $\tau \in[0, T_0]$}.
	\end{equation}
	Then
	\begin{equation}
		h(\tau)\leq \frac{1}{[\ep(1+\delta)]^{\frac{1}{\delta}}}\equiv h_0
		\ \ \mbox{for each $\tau\in [0, T_0]$}\nonumber
	\end{equation}  whenever
	\begin{equation}\label{f1}
		\ep\leq \frac{\delta^\delta}{(b+\delta)^\delta(1+\delta)^{1+\delta}}\ \ \mbox{and} \ \ h(0)\leq
		h_0.
	\end{equation}
\end{lemma}
We will use the proof of this lemma which is given in \cite{X4}.

The following lemma can be found in (\cite{D}, p.12).
\begin{lemma}\label{ynb}
	Let $\{y_n\}, n=0,1,2,\cdots$, be a sequence of positive numbers satisfying the recursive inequalities
	\begin{equation*}
	y_{n+1}\leq cb^ny_n^{1+\alpha}\ \ \mbox{for some $b>1, c, \alpha\in (0,\infty)$.}
	\end{equation*}
	If
	\begin{equation*}
	y_0\leq c^{-\frac{1}{\alpha}}b^{-\frac{1}{\alpha^2}},
	\end{equation*}
	then $\lim_{n\rightarrow\infty}y_n=0$.
\end{lemma}

\begin{lemma}\label{gub}
	Let $u$ be a sub-solution of the problem
	\begin{eqnarray}
		\pt u-D^2\Delta u&=&gu+f+\mdiv\f\ \ \mbox{in $\ot$},\label{u1}\\
		u&=&0\ \ \mbox{on $\Sigma_T$,}\label{u3}\\
		u&=&u_0\ \ \mbox{on $\Omega$},\label{u2}
	\end{eqnarray}
where 
$u_0, g,f,\f$ are given functions in suitable function spaces.
Then for each 
\begin{equation}\label{u3}
q>1+\frac{N}{2}
\end{equation}
there exists a positive number $c=c(D, N, \Omega, q)$ such that
\begin{equation}\label{u6}
\sup_{\ot}u\leq 2\sup_\Omega u_0+ c\left(\|g\|_{q,\ot}^{\frac{(q-1)s_0}{2q}}+1\right)\|u^+\|_{\frac{2q}{q-1},\ot}+\|f\|_{q,\ot}T^{\frac{1}{2s_0}}+\|\f\|_{2q,\ot},	
\end{equation}
where
\begin{equation}
		s_0=\frac{q(N+2)}{2q-N-2}.\label{sdef}
\end{equation}
\end{lemma}
This lemma is essentially known. The interest here lies in the fact that it gives the precise dependence of the uniform up bound on $T$ and the given functions. This is very important to our late development.
\begin{proof}The proof is based upon the De Giorgi iteration scheme. 
	Let 
	\begin{equation}\label{conk2}
		k\geq 2\sup_\Omega u_0
	\end{equation} be selected as below. Define
	$$k_n=k-\frac{k}{2^{n+1}},\ \ n=0,1,2,\cdots.$$	
Use $(u-k_{n+1})^+$ as a test function in \eqref{u1} and use the fact that  $\left.(u-k_{n+1})^+\right|_{t=0}=0$ to get
\begin{eqnarray}
\lefteqn{	\frac{1}{2}\sup_{0\leq t\leq T}\io\left[(u-k_{n+1})^+\right]^2dx+D^2\ioT\left|\nabla(u-k_{n+1})^+\right|^2dxdt}\nonumber\\
&\leq&2\ioT gu(u-k_{n+1})^+dxdt+2\ioT f(u-k_{n+1})^+dxdt-2\ioT\f\cdot\nabla(u-k_{n+1})^+dxdt.\label{gub1}
\end{eqnarray}
Set
$$Q_n=\{(x,t)\in\Omega_T:u(x,t)\geq k_n\}.$$
Then we easily see from Young's inequality (\cite{GT}, p.145) that
\begin{eqnarray}
	\left|\ioT\f\cdot\nabla(u-k_{n+1})^+dxdt\right|&\leq&\ve\ioT\left|\nabla(u-k_{n+1})^+\right|^2dxdt+\frac{c}{\ve}\|\f\|_{2,Q_{n+1}}^2,\ \ \ve>0.\label{u10}
\end{eqnarray}
By H\"{o}lder's inequality, 
\begin{eqnarray}
	\left|\ioT f(u-k_{n+1})^+dxdt\right|&\leq&
	\|f\|_{\frac{2(N+2)}{N+4},Q_{n+1}}\|(u-k_{n+1})^+\|_{\frac{2(N+2)}{N},\ot}.\label{u12}
\end{eqnarray}
Note that
\begin{eqnarray}
\left[(u-k_{n})^+\right]^2&\geq&(u-k_{n})^+(u-k_{n+1})^+\nonumber\\
&\geq & u\left(1-\frac{k_n}{k_{n+1}}\right)(u-k_{n+1})^+\nonumber\\
&\geq &\frac{1}{2^{n+2}}u(u-k_{n+1})^+.\nonumber
\end{eqnarray}
Consequently,
\begin{eqnarray}
	\left|\ioT gu(u-k_{n+1})^+dxdt\right|&\leq& 2^{n+2}\ioT |g|\left[(u-k_{n})^+\right]^2dxdt\nonumber\\
	&\leq& 2^{n+2}\|g\|_{q,\ot}y_n^{\frac{q-1}{q}},\label{u11}
\end{eqnarray}
where
$$y_n=\ioT\left[(u-k_{n})^+\right]^{\frac{2q}{q-1}}dxdt.$$
Substitute \eqref{u10}, \eqref{u12}, and \eqref{u11} into \eqref{gub1} and choose $\ve$ suitably small in the resulting inequality to derive
 \begin{eqnarray}
 	\lefteqn{\sup_{0\leq t\leq T}\io\left[(u-k_{n+1})^+\right]^2dx+\ioT\left|\nabla(u-k_{n+1})^+\right|^2dxdt}\nonumber\\
 	&\leq& c2^{n}\|g\|_{q,\ot}y_n^{\frac{q-1}{q}}+c\|f\|_{\frac{2(N+2)}{N+4},Q_{n+1}}\|(u-k_{n+1})^+\|_{\frac{2(N+2)}{N},\ot}+c\|\f\|_{2,Q_{n+1}}^2.\label{gub3}
 \end{eqnarray}
It follows from Poincar\'{e}'s inequality that
\begin{eqnarray}
\lefteqn{	\ioT\left[(u-k_{n+1})^+\right]^{2+\frac{4}{N}}dxdt}\nonumber\\
&\leq&\int_{0}^{T}\left(\io\left[(u-k_{n+1})^+\right]^2dx\right)^{\frac{2}{N}}\left(\io\left[(u-k_{n+1})^+\right]^{\frac{2N}{N-2}}dx\right)^{\frac{N-2}{N}}dt\nonumber\\
&\leq&c\left(\sup_{0\leq t\leq T}\io\left[(u-k_{n+1})^+\right]^2dx\right)^{\frac{2}{N}}\ioT\left|\nabla(u-k_{n+1})^+\right|^2dxdt\nonumber\\
&\leq&c\left(2^{n}\|g\|_{q,\ot}y_n^{\frac{q-1}{q}}+\|f\|_{\frac{2(N+2)}{N+4},Q_{n+1}}\|(u-k_{n+1})^+\|_{\frac{2(N+2)}{N},\ot}+\|\f\|_{2,Q_{n+1}}^2\right)^{\frac{N+2}{N}}\nonumber\\
&\leq&c2^{\frac{n(N+2)}{N}}\|g\|_{q,\ot}^{\frac{N+2}{N}}y_n^{\frac{(q-1)(N+2)}{qN}}+c\|f\|_{\frac{2(N+2)}{N+4},Q_{n+1}}^{\frac{N+2}{N}}\|(u-k_{n+1})^+\|_{\frac{2(N+2)}{N},\ot}^{\frac{N+2}{N}}+c\|\f\|_{2,Q_{n+1}}^{\frac{2(N+2)}{N}}\nonumber\\
&\leq&\varepsilon\ioT\left[(u-k_{n+1})^+\right]^{2+\frac{4}{N}}dxdt+\frac{c}{\varepsilon}\|f\|_{\frac{2(N+2)}{N+4},Q_{n+1}}^{\frac{2(N+2)}{N}}+c\|\f\|_{2,Q_{n+1}}^{\frac{2(N+2)}{N}}\nonumber\\
&&+c2^{\frac{n(N+2)}{N}}\|g\|_{q,\ot}^{\frac{N+2}{N}}y_n^{\frac{(q-1)(N+2)}{qN}},\ \ \varepsilon>0.
\end{eqnarray}
Consequently,
\begin{equation}\label{gub2}
	\ioT\left[(u-k_{n+1})^+\right]^{2+\frac{4}{N}}dxdt\leq c2^{\frac{n(N+2)}{N}}\|g\|_{q,\ot}^{\frac{N+2}{N}}y_n^{\frac{(q-1)(N+2)}{qN}}+c\|f\|_{\frac{2(N+2)}{N+4},Q_{n+1}}^{\frac{2(N+2)}{N}}+c\|\f\|_{2,Q_{n+1}}^{\frac{2(N+2)}{N}}.
\end{equation}
The last two terms in the above inequality can be estimated as follows:
\begin{eqnarray}
\|f\|_{\frac{2(N+2)}{N+4},Q_{n+1}}^{\frac{2(N+2)}{N}}&=&\left(\int_{Q_{n+1}}|f|^{\frac{2(N+2)}{N+4}}dxdt\right)^{\frac{N+4}{N}}\nonumber\\
	&\leq&	\|f\|_{q,\ot}^{\frac{2(N+2)}{N}}\left|Q_{n+1}\right|^{\frac{N+4}{N}-\frac{2(N+2)}{Nq}},
	\nonumber\\
	\|\f\|_{2,Q_{n+1}}^{\frac{2(N+2)}{N}}&\leq&\|\f\|_{2q,\ot}^{\frac{2(N+2)}{N}}\left|Q_{n+1}\right|^{\frac{N+2}{N}-\frac{N+2}{Nq}}.
\end{eqnarray}
Use this in \eqref{gub2} to obtain
\begin{eqnarray}
\ioT\left[(u-k_{n+1})^+\right]^{2+\frac{4}{N}}dxdt
&	\leq& c2^{\frac{n(N+2)}{N}}\|g\|_{q,\ot}^{\frac{N+2}{N}}y_n^{\frac{(q-1)(N+2)}{qN}}+c\|f\|_{q,\ot}^{\frac{2(N+2)}{N}}\left|Q_{n+1}\right|^{\frac{N+4}{N}-\frac{2(N+2)}{Nq}}\nonumber\\
&&+c\|\f\|_{2q,\ot}^{\frac{2(N+2)}{N}}\left|Q_{n+1}\right|^{\frac{N+2}{N}-\frac{N+2}{Nq}}.\label{u4}
\end{eqnarray}
By \eqref{u3},
$$\frac{q}{q-1}<\frac{N+2}{N}.$$
This together with \eqref{u4} implies
\begin{eqnarray}
	y_{n+1}&=&\ioT\left[(u-k_{n+1})^+\right]^{\frac{2q}{q-1}}dxdt\nonumber\\ &\leq&\left(\ioT\left[(u-k_{n+1})^+\right]^{2+\frac{4}{N}}dxdt\right)^{\frac{qN}{(q-1)(N+2)}}\left|Q_{n+1}\right|^{1-\frac{qN}{(q-1)(N+2)}}\nonumber\\
	&\leq&c2^{\frac{qn}{q-1}}\|g\|_{q,\ot}^{\frac{q}{q-1}}y_n\left|Q_{n+1}\right|^{1-\frac{qN}{(q-1)(N+2)}}\nonumber\\
	&&+c\|f\|_{q,\ot}^{\frac{2q}{q-1}}\left|Q_{n+1}\right|^{\frac{(N+4)q-2(N+2)}{(q-1)(N+2)}+1-\frac{qN}{(q-1)(N+2)}}\nonumber\\
	&&+c\|\f\|_{2q,\ot}^{\frac{2q}{q-1}}\left|Q_{n+1}\right|^{\frac{(N+2)q-N-2}{(q-1)(N+2)}+1-\frac{qN}{(q-1)(N+2)}}\nonumber\\
	&=&c2^{\frac{qn}{q-1}}\|g\|_{q,\ot}^{\frac{q}{q-1}}y_n\left|Q_{n+1}\right|^{\alpha}+c\|f\|_{q,\ot}^{\frac{2q}{q-1}}\left|Q_{n+1}\right|^{1+2\alpha}+c\|\f\|_{2q,\ot}^{\frac{2q}{q-1}}\left|Q_{n+1}\right|^{1+\alpha},\label{nub3}
\end{eqnarray}
where
\begin{equation}\label{adef}
\alpha=\frac{2q-N-2}{(q-1)(N+2)}=\frac{q}{(q-1)s_0}>0.	
\end{equation}
We easily see that
$$y_n\geq\int_{Q_{n+1}}(k_{n+1}-k_n)^{\frac{2q}{q-1}}dxdt=\frac{k^{\frac{2q}{q-1}}}{2^{\frac{2q(n+2)}{q-1}}}\left|Q_{n+1}\right|.$$
Therefore,
\begin{eqnarray}
\left|Q_{n+1}\right|^{\alpha}
	&\leq &\frac{2^{\frac{2q(n+2)\alpha}{q-1}}}{k^{\frac{2q\alpha}{q-1}}}y_n^{\alpha},\\
	\left|Q_{n+1}\right|^{1+\alpha}
	&\leq &\frac{2^{\frac{2q(n+2)(1+\alpha)}{q-1}}}{k^{\frac{2q(1+\alpha)}{q-1}}}y_n^{1+\alpha},\\
	\left|Q_{n+1}\right|^{1+2\alpha}&\leq&\left|Q_{n+1}\right|^{1+\alpha}|\ot|^{\alpha}
	\leq\frac{2^{\frac{2q(n+2)}{q-1}\left(1+\alpha\right)}}{k^{\frac{2q}{q-1}\left(1+\alpha\right)}}|\ot|^{\alpha}y_n^{1+\alpha}.
\end{eqnarray}
Collecting the preceding three estimates in \eqref{nub3} to get
\begin{eqnarray}
	y_{n+1}&\leq&\frac{c2^{\frac{(2\alpha+1)qn}{q-1}}\|g\|_{q,\ot}^{\frac{q}{q-1}}}{k^{\frac{2q\alpha}{q-1}}}y_n^{1+\alpha}\nonumber\\
	&&+\frac{c2^{\frac{2(\alpha+1)qn}{q-1}}\left(\|f\|_{q,\ot}^{\frac{2q}{q-1}}T^{\alpha}+\|\f\|_{2q,\ot}^{\frac{2q}{q-1}}\right)}{k^{\frac{2q}{q-1}\left(1+\alpha\right)}}y_n^{1+\alpha}.\label{nub4}
\end{eqnarray}
We choose $k$ so large that
\begin{equation}\label{kcon}
	\frac{\|f\|_{q,\ot}^{\frac{2q}{q-1}}T^{\alpha}+\|\f\|_{2q,\ot}^{\frac{2q}{q-1}}}{k^{\frac{2q}{q-1}}}\leq 1.
\end{equation}
Use this in \eqref{nub4} to get
\begin{equation}
		y_{n+1}\leq\frac{c2^{\frac{2(\alpha+1)qn}{q-1}}\left(\|g\|_{q,\ot}^{\frac{q}{q-1}}+1\right)}{k^{\frac{2q\alpha}{q-1}}}y_n^{1+\alpha}
\end{equation}
According Lemma \ref{ynb}, if we further require $k$ to satisfies
$$y_0\leq c\left(\frac{k^{\frac{2q\alpha}{q-1}}}{c\left(\|g\|_{q,\ot}^{\frac{q}{q-1}}+1\right)}\right)^{\frac{1}{\alpha}},$$
then
\begin{equation}\label{u5}
	\sup_{\ot} u\leq k.
\end{equation}
In view of \eqref{kcon} and \eqref{conk2}, it is enough for us to take 
\begin{equation}\label{kdef}
k=2\sup_{\Omega} u_0+c\left(\|g\|_{q,\ot}^{\frac{1}{2\alpha}}+1\right)y_0^{\frac{q-1}{2q}}+\|f\|_{q,\ot}T^{\frac{1}{2s_0}}+\|\f\|_{2q,\ot}.	
\end{equation}
Note that 
$$y_0^{\frac{q-1}{2q}}\leq\left(\ioT\left(u^+\right)^{\frac{2q}{q-1}}\right)^{\frac{q-1}{2q}}=\|u^+\|_{\frac{2q}{q-1},\ot}.$$
This combined with \eqref{u5} and \eqref{adef} gives \eqref{u6}. The proof is complete.
	\end{proof}
If
$$\sup_{\ot}u=\|u\|_{\infty,\ot},$$
we can apply the interpolation inequality (\cite{GT},p.146) to obtain
$$\|u^+\|_{\frac{2q}{q-1},\ot}\leq\|u\|_{\frac{2q}{q-1},\ot}\leq \ve\|u\|_{\infty,\ot}+\frac{1}{\ve^{\frac{q+1}{q-1}}}\|u\|_{1,\ot},\ \ \ve>0.$$
Plug this into \eqref{u6} to get
\begin{eqnarray}
\|u\|_{\infty,\ot}	&\leq& c\left(\|g\|_{q,\ot}^{\frac{1}{2\alpha}}+1\right)\left(\ve\|u\|_{\infty,\ot}+\frac{1}{\ve^{\frac{q+1}{q-1}}}\|u\|_{1,\ot}\right)\nonumber\\
&&+2\sup_{\Omega} u_0+\|f\|_{q,\ot}T^{\frac{1}{2s_0}}+\|\f\|_{2q,\ot}.\nonumber	
\end{eqnarray}
Take $\ve$ so that the coefficient of the term $\|u\|_{\infty,\ot}$ on the right hand side of the above inequality
$$c\left(\|g\|_{q,\ot}^{\frac{1}{2\alpha}}+1\right)\ve=\frac{1}{2}$$
to drive
\begin{equation}\label{u7}
	\|u\|_{\infty,\ot}\leq 4\sup_{\Omega} u_0+ c\left(\|g\|_{q,\ot}^{s_0}+1\right)\|u\|_{1,\ot}+2\|f\|_{q,\ot}T^{\frac{1}{2s_0}}+2\|\f\|_{2q,\ot}.
\end{equation}
Here we have used \eqref{adef}.
\begin{lemma}\label{nmub1} Assume that (H3) holds.
	Let $u$ be the solution of the problem
	\begin{eqnarray}
		\pt u-D^2\Delta u&=&f\ \ \ \mbox{in $\ot$,}\\
		u&=&0\ \ \mbox{on $\Sigma_T$,}\\
		u&=&u_0\ \ \mbox{on $\Omega$},
	\end{eqnarray}
where
\begin{equation}
	u_0\in W^{1,\infty}_0(\Omega),\ \ f\in L^{2q}(\ot)\ \ \mbox{for some $q>1+\frac{N}{2}$}
\end{equation}
Then there is a positive number $c=c(\Omega, N, q)$ such that
\begin{equation}\label{nuub}
\|\nabla u\|_{\infty,\ot}\leq c\|\nabla u_0\|_{\infty,\Omega}+c\|f\|_{2q,\ot}.
\end{equation}
\end{lemma}
This lemma is known. In fact, it is not difficult for us to see \eqref{nuub}. Indeed, $u_{x_i}$ satisfies \eqref{u1} with $f$ being replaced by $f_{x_i}$ and $g,\g$ being $0$. A full proof can be inferred from Proposition 2.3 in \cite{X4}.
\section{$W^{1,q}$ estimates for elliptic equations}

Before we prove Theorem \ref{w1q}, we recall some results from \cite{CFL}.

\begin{definition}
	A function $k(x)$ on $\RN\setminus\{0\}$ is called a Calder\'{o}n-Zygmund kernel (in short, C-Z kernel) if:
	\begin{enumerate}
		\item[(i)] $k\in C^\infty(\RN\setminus\{0\})$;
		\item[(ii)]$k(x)$ is homogeneous of degree $-N$, i.e., $k(tx)=t^{-N}k(x)$;
		\item[(iii)] $\int_{\partial B_1(0)}k(x)d\mathcal{H}^{N-1}=0$.
	\end{enumerate}
\end{definition}
The most fundamental result concerning C-Z kernels \cite{CFL} is the following
\begin{lemma}
	Given a C-Z kernel $k(x)$,  we define
	\begin{equation}
		K_\varepsilon f(x)=\int_{\RN\setminus B_\varepsilon(x)}k(x-y)f(y)dy\ \ \mbox{for $\varepsilon>0$ and $f\in L^q(\RN)$ with $q\in (1,\infty)$.}\nonumber
	\end{equation}
	Then:
	\begin{enumerate}
		\item[\textup{(CZ1)}] For each $f\in L^q(\RN)$ there exists a function $Kf\in L^q(\RN)$ such that
		\begin{equation}
			\lim_{\varepsilon\rightarrow 0}\|K_\varepsilon f-Kf\|_{q,\RN}=0.\nonumber
		\end{equation}
		In this case we use the notation
		\begin{equation}
			Kf(x)=\textup{P.V.}k*f(x)=\textup{P.V.}\int_{\RN}k(x-y)f(y)dy.\nonumber
		\end{equation}
		\item[\textup{(CZ2)}] $K$ is a bounded operator on $ L^q(\RN)$. More precisely, we have
		\begin{equation}
			\|Kf\|_{q,\RN}\leq c\left(\int_{\partial B_1(0)}k^2(x)d\mathcal{H}^{N-1}\right)^\frac{1}{2}	\|f\|_{q,\RN},\nonumber
		\end{equation}
		where the positive number $c$ depends only on $N, q$.
	\end{enumerate}
\end{lemma}

We are ready to prove Theorem \ref{w1q}.
\begin{proof}[Proof of Theorem \ref{w1q}] As in \cite{CFL,DI}, the proof comprises a local interior estimate and a boundary estimate.
	To establish the former, we fix $x_0\in\Omega$. Let $0<\delta<r$ with $B_r(x_0)\subset\Omega$. Pick a smooth cutoff function $\xi\in C_0^\infty(\RN)$ such that
\begin{eqnarray*}
	\xi&=& 1\ \ \textup{on $B_{\delta}(x_0)$},\\
	\xi&=&0\ \ \textup{outside $B_{r}(x_0)$,}\\
	0\leq\xi&\leq&1\ \ \textup{on $B_{r}(x_0)$, }\\
	|\nabla\xi|&\leq& \frac{c}{r-\delta}\ \ \textup{on $B_{r}(x_0)$. }
\end{eqnarray*}
Set
\begin{equation*}
	u=p\xi.
\end{equation*}
We can easily verify that $u$ satisfies the equation
\begin{eqnarray}
	-\textup{div}\left[(I+\w(x)\otimes\w(x))\nabla u\right]&=&-\textup{div}\left[p(I+\w(x)\otimes\w(x))\nabla \xi\right]+F
	\ \ \textup{in $\RN$,}\label{eqnu}
\end{eqnarray}
where
\begin{equation}\label{fdef}
	F=\xi S(x)-\nabla\xi(I+\w(x)\otimes\w(x))\nabla p.
\end{equation}
Set
\begin{eqnarray}
	A(x_0)=I+\air\w(x)dx\otimes\air\w(x)dx.
\end{eqnarray}
Then we can write the above equation in the form
\begin{eqnarray}
	\lefteqn{-\mbox{div}[A(x_0)\nabla u(x)	] }\nonumber\\
	&=&\mbox{div}\left[((\w(x)-\air\w(x)dx )\cdot\nabla u(x))\w(x)\right]\nonumber\\
	&&+\mbox{div}\left[(\air\w(x)dx \cdot\nabla u(x))(\w(x)-\air\w(x)dx )\right]\nonumber\\
	&&-\textup{div}\left[p(x)(I+\w(x)\otimes\w(x))\nabla \xi(x)\right] +F(x),\label{ueqn}
\end{eqnarray}
Recall that the fundamental solution of the equation $-\mbox{div}[A(x_0)\nabla v(x)	]=0$ is given by
\begin{equation}
	\Gamma(x_0,x)=\frac{1}{(N-2)\omega_N\sqrt{\textup{det}A(x_0)}} \left(\sum_{i,j=1}^{N}A_{ij}(x_0)x_ix_j\right)^{\frac{2-N}{2}},\nonumber
\end{equation}
where $A_{ij}(x_0)$ is the co-factor of the entry that lies in the $i$th row and the $j$th column in the matrix $A(x_0)$ and $\omega_N$ is the surface area of the unit sphere. Then we have the frequently used representation formula
\begin{equation}\label{rep}
	u(y)=-\int_{ B_r(x_0)}\Gamma(x_0,x-y)\mbox{div}[A(x_0)\nabla u(x)	]dx
\end{equation}
whenever $u$ is compactly supported in $B_r(x_0)$.
We can easily verify that $\frac{\partial^2}{\partial x_i\partial x_j}\Gamma(x_0,x), i,j=1, \cdots, N$  are C-Z kernels \cite{CFL,DI}. Our key observation is the following
\begin{lemma} 
There is a positive number $c=c(N)$ such that
	\begin{equation}
		\left|\frac{\partial^2\Gamma(x_0,x)}{\partial x_i\partial x_j}\right|\leq c\left(1+\left|\air\w(x)dx \right|^2\right)^{\frac{5N-2}{2}}\leq c(1+\|\w\|_{\infty,\Omega})^{5N-2} \ \ \textup{for $x\in \partial B_1(0)$.}\nonumber
	\end{equation}
\end{lemma}
\begin{proof}
	Obviously,
	\begin{equation}\label{mze}
		|y|^2\leq	\left(I+\air\w(x) dx\otimes\air\w(x)dx\right)y\cdot y\leq \left(1+\left|\air\w(x)dx\right|^2\right)|y|^2\ \ \mbox{for all $y\in\RN$.}
	\end{equation}
%
	This implies
	\begin{eqnarray}
		\mbox{the largest eigenvalue of $A(x_0)$}&\leq &1+\left|\air\w(x)dx \right|^2,\label{leigen}\\
		\mbox{the smallest eigenvalue of $A(x_0)$}&\geq &1.\label{seigen}
	\end{eqnarray} 
	Hence,
	\begin{eqnarray}
			\textup{det}A(x_0)
		&=&\mbox{the product of the eigenvalues, counted with multiplicity}\nonumber\\
		&\in& \left(1, \left(1+\left|\air\w(x)dx \right|^2\right)^N\right).\label{alei}
	\end{eqnarray}
	We calculate
	\begin{eqnarray}
		\partial_{x_k}\Gamma(x_0,x)&=&-\frac{\sum_{i=1}^{N}(A_{ik}(x_0)+A_{ki}(x_0))x_i}{2\omega_N\sqrt{\textup{det}A(x_0)}\left(\sum_{i,j=1}^{N}A_{ij}(x_0)x_ix_j\right)^{\frac{N}{2}}} ,\label{fop}\\
		\partial^2_{x_kx_\ell}\Gamma(x_0,x)&=&-\frac{A_{\ell k}(x_0)+A_{k\ell}(x_0)}{2\omega_N\sqrt{\textup{det}A(x_0)}\left(\sum_{i,j=1}^{N}A_{ij}(x_0)x_ix_j\right)^{\frac{N}{2}}} \nonumber\\
		&&+\frac{N\sum_{i=1}^{N}(A_{ik}(x_0)+A_{ki}(x_0))x_i\sum_{j=1}^{N}(A_{j\ell}(x_0)+A_{\ell j}(x_0))x_\ell}{4\omega_N\sqrt{\textup{det}A(x_0)}\left(\sum_{i,j=1}^{N}A_{ij}(x_0)x_ix_j\right)^{\frac{N+2}{2}}}\label{sop}.
	\end{eqnarray}
	Note that
	\begin{equation}
		A^{-1}(x_0)=\frac{1}{\textup{det}A(x_0)}\left(\begin{array}{clc}
			A_{11}&\cdots &A_{N1}\\
			\vdots&      &\vdots\\
			A_{1N}&\cdots&A_{NN}
		\end{array}\right).\nonumber
	\end{equation}
	It is easy to see that
	\begin{eqnarray}
		\sum_{i,j=1}^{N}A_{ij}(x_0)x_ix_j=\textup{det}A(x_0)A^{-1}(x_0)x\cdot x.\nonumber
	\end{eqnarray}
	Recall that $\lambda$ is an eigenvalue of an invertible matrix $A$ if and only $\frac{1}{\lambda}$ is an eigenvalue of $A^{-1}$. With this in mind,  we derive from \eqref{leigen}, \eqref{seigen}, and \eqref{alei} that 
	\begin{eqnarray}
		\sum_{i,j=1}^{N}A_{ij}(x_0)x_ix_j&\leq&\textup{det}A(x_0)|x|^2\leq \left(1+\left|\air\w(x)dx \right|^2\right)^N|x|^2,\nonumber\\
		\sum_{i,j=1}^{N}A_{ij}(x_0)x_ix_j&\geq&\frac{\textup{det}A(x_0)}{1+\left|\air\w(x)dx \right|^2}|x|^2\geq\frac{1}{1+\left|\air\w(x)dx \right|^2}|x|^2.\label{ai}
	\end{eqnarray}
	Remember that the determinant of an n × n matrix A is the signed sum over all possible products of n entries of A with exactly one entry being selected from each row and from each column of A. Thus,
	\begin{equation}\label{cofa}
		|A_{ij}(x_0)|\leq (N-1)!\left(1+\left|\air\w(x)dx \right|^2\right)^{N-1}.
	\end{equation}
	Here we have used the fact that each entry in $A(x_0)$ is bounded by $1+\left|\air\w(x)dx \right|^2$. We are ready to estimate
	for $x\in \partial B_1(0)$ that
	\begin{eqnarray}
		\left|	\partial^2_{x_kx_\ell}\Gamma(x_0,x)\right|&\leq&c\left(1+\left|\air\w(x)dx \right|^2\right)^{N-1+\frac{N}{2}}\nonumber\\
		&&+c\left(1+\left|\air\w(x)dx \right|^2\right)^{2(N-1)+\frac{N+2}{2}}\nonumber\\
		&\leq &c\left(1+\left|\air\w(x)dx \right|^2\right)^{\frac{5N-2}{2}}.\label{funb}
	\end{eqnarray}
	The proof is complete.
\end{proof} 
Return to the proof of Lemma \ref{w1q}. Since $u$ is compactly supported in $B_r(x_0)$, we have from \eqref{rep} that
\begin{eqnarray}
	u(y)&=&\int_{B_r(x_0)}((\w(x)-\air\w(x)dx )\cdot\nabla u(x))\w(x)\cdot\nabla_x\Gamma(x_0, y-x)dx\nonumber\\
	&&+\int_{B_r(x_0)}(\air\w(x)dx \cdot\nabla u(x))(\w(x)-\air\w(x)dx )\cdot\nabla_x\Gamma(x_0, y-x)dx\nonumber\\
	&&+\int_{B_r(x_0)}p(x)(I+\w(x)\otimes\w(x))\nabla \xi(x) \cdot\nabla_x\Gamma(x_0, y-x)dx\nonumber\\
	&&+\int_{B_r(x_0)}F(x)\Gamma(x_0, y-x)dx.
\end{eqnarray}
Differentiate the above equation with respect to $y_i$ to obtain
\begin{eqnarray}
	u_{y_i}(y)&=&-\int_{B_r(x_0)}((m_h(x)-m_h(x_0)) u_{x_h}(x))m_j(x)\frac{\partial^2}{\partial x_i\partial x_j}\Gamma(x_0, y-x)dx\nonumber\\
	&&-\int_{B_r(x_0)}(m_h(x_0) u_{x_h}(x))(m_j(x)-m_j(x_0))\frac{\partial^2}{\partial x_i\partial x_j}\Gamma(x_0, y-x)dx\nonumber\\
	&&-\int_{B_r(x_0)}p(x)(I+\w(x)\otimes\w(x))\nabla \xi(x) \cdot\nabla_x\Gamma_{x_i}(x_0, y-x)dx\nonumber\\
	&&-\int_{B_r(x_0)}F(x)\Gamma_{x_i}(x_0, y-x)dx. \label{rep1}
\end{eqnarray}
According to Morrey's inequality (\cite{EG}, p.143), for  $\ell>N$ there is a positive number $c=c(N, \ell)$ such that
\begin{equation}\label{lip}
	\left|\w(x)-\air\w(x)dx \right|\leq cr^{1-\frac{N}{\ell}}\|\nabla\w\|_{\ell, B_r(x_0)}\ \ \mbox{on $\br$}. 
\end{equation}
With this in mind, we apply (CZ2) to \eqref{rep1} to deduce
\begin{eqnarray}
	\lefteqn{\|\nabla u\|_{q, B_{r }(x_0)}}\nonumber\\
	&\leq&cr^{1-\frac{N}{\ell}}\sum_{i,j=1}^{N}\left\|\frac{\partial^2}{\partial x_i\partial x_j}\Gamma(x_0, x)\right\|_{2,\partial B_1(0)}\|\w\|_{\infty, B_{ r}(x_0)}\|\nabla\w\|_{\infty,B_r(x_0)}\|\nabla u\|_{q, B_{r }(x_0)}\nonumber\\
	&&+\frac{c}{r-\delta}\sum_{i,j=1}^{N}\left\|\frac{\partial^2}{\partial x_i\partial x_j}\Gamma(x_0, x)\right\|_{2,\partial B_1(0)}(1+\|\w\|_{\infty, B_{ r}(x_0)}^2)\|p\|_{q, B_{ r}(x_0)}\nonumber\\
	&&+c\left\|\int_{B_r(x_0)}F(x)\Gamma_{x_i}(y, y-x)dx\right\|_{q,B_{r }(x_0) }\nonumber\\
	&\leq&cr^{1-\frac{N}{\ell}}(1+\|\w\|_{\infty, B_{ r}(x_0)})^{5N-1}\|\nabla\w\|_{\infty, B_{ r}(x_0)}\|\nabla u\|_{q, B_{ r}(x_0)}\nonumber\\
	&&+\frac{c}{r-\delta}\|p\|_{q, B_{ r}(x_0)}\left(1+\|\w\|_{\infty, B_{ r}(x_0)}\right)^{5N}+\left\|\int_{B_r(x_0)}F(x)\Gamma_{x_i}(x_0, y-x)dx\right\|_{q, B_{r }(x_0)}.\label{rep2}
\end{eqnarray}
The last step is due to \eqref{funb}. To estimate the last term in the above inequality, we derive from \eqref{fop}, \eqref{ai}, and \eqref{cofa} that
\begin{eqnarray}
	\left|\Gamma_{x_i}(x_0, y-x)\right|&\leq& \frac{c\left(1+\left|\air\w(x)dx \right|^2\right)^{\frac{3N-2}{2}}}{|x-y|^{N-1}}.
\end{eqnarray}
Consequently,
\begin{eqnarray}
	\lefteqn{\left|\int_{B_r(x_0)}F(x)\Gamma_{x_i}(x_0, y-x)dx\right|}\nonumber\\
	&\leq &c\left(1+\|\w\|_{\infty, B_{ r}(x_0)}\right)^{3N-2}\int_{B_r(x_0)}\frac{|\xi S(x)|+|\nabla\xi|\left(1+\|\w\|^2_{\infty, B_{ r}(x_0)}\right)|\nabla p|}{|x-y|^{N-1}}dx.
\end{eqnarray}
By the remark following Lemma 7.12 in \cite{GT}, we obtain
\begin{eqnarray}
	\left\|\int_{B_r(x_0)}F(x)\Gamma_{x_i}(x_0, y-x)dx\right\|_{q,B_{ r}(x_0) }&\leq&c\left(1+\|\w\|_{\infty, B_{ r}(x_0)}\right)^{3N-2}\|S\|_{\frac{Nq}{N+q}, B_r(x_0)}\nonumber\\
	&&+\frac{c\left(1+\|\w\|_{\infty, B_{ r}(x_0)}\right)^{3N}}{r-\delta}\|\nabla p\|_{\frac{Nq}{N+q}, B_r(x_0)}.
\end{eqnarray}
Plug this into \eqref{rep2} and then use the definition of $\xi$ to obtain
\begin{eqnarray}
	\|\nabla p\|_{q, B_{\delta }(x_0)}&\leq&cr^{1-\frac{N}{\ell}}(1+\|\w\|_{\infty, B_{ r}(x_0)})^{5N-1}\|\nabla\m\|_{q, B_{ r}(x_0)}\|\nabla p\|_{q, B_{ r}(x_0)}\nonumber\\
	&&+\frac{cr^{1-\frac{N}{\ell}}}{r-\delta}(1+\|\w\|_{\infty, B_{ r}(x_0)})^{5N-1}\|\nabla\m\|_{q, B_{ r}(x_0)}\|p\|_{q, B_{ r}(x_0)}\nonumber\\
	&&+\frac{c}{r-\delta}\left(1+\|\w\|_{\infty, B_{ r}(x_0)}\right)^{5N}\|p\|_{q, B_{ r}(x_0)}+c\left(1+\|\w\|_{\infty, B_{ r}(x_0)}\right)^{3N-2}\|S\|_{\frac{Nq}{N+q}, B_r(x_0)}\nonumber\\
	&&+\frac{c}{r-\delta}\left(1+\|\w\|_{\infty, B_{ r}(x_0)}\right)^{3N}\|\nabla p\|_{\frac{Nq}{N+q}, B_r(x_0)}.\label{rep3}
\end{eqnarray}
Set
\begin{eqnarray}
	K_1&=&r^{1-\frac{N}{\ell}}(1+\|\w\|_{\infty, B_{ r}(x_0)})^{5N-1}\|\nabla\m\|_{q, B_{ r}(x_0)}
	,\label{rep5}\\
	K_2&=&r^{1-\frac{N}{\ell}}(1+\|\w\|_{\infty, B_{ r}(x_0)})^{5N-1}\|\nabla\m\|_{q, B_{ r}(x_0)}\|p\|_{q, B_{ r}(x_0)}\nonumber\\
	&&+\left(1+\|\w\|_{\infty, B_{ r}(x_0)}\right)^{5N}\|p\|_{q, B_{ r}(x_0)}+\left(1+\|\w\|_{\infty, B_{ r}(x_0)}\right)^{3N}\|\nabla p\|_{\frac{Nq}{N+q}, B_r(x_0)},\\
	K_3&=&\left(1+\|\w\|_{\infty, B_{ r}(x_0)}\right)^{3N-2}\|S\|_{\frac{Nq}{N+q}, B_r(x_0)}.
\end{eqnarray}
We can write \eqref{rep3} as
\begin{equation}
	\|\nabla p\|_{q, B_{\delta }(x_0)}\leq cK_1\|\nabla p\|_{q, B_{ r}(x_0)}
	+\frac{c}{r-\delta}K_2+cK_3.\label{rep4}
\end{equation}
Define
$$r_n=r-\frac{r}{2^{n+1}},\ \ n=0,1,2,\cdots.$$
Take $(\delta,r)=( r_{n}, r_{n+1})$ in \eqref{rep4} and keep in mind the fact that $K_1, K_2$, and $K_3$ are all increasing with $r$  to  get
\begin{eqnarray}
	\|\nabla p\|_{q, B_{r_n }(x_0)}&\leq& cK_1\|\nabla p\|_{q, B_{ r_{n+1}}(x_0)}+\frac{c}{r_{n+1}-r_n}K_2+cK_3\nonumber\\
		&\leq&c K_1\|\nabla p\|_{q, B_{ r_{n+1}}(x_0)}+\frac{c2^{n+2}}{r}K_2+cK_3.
\end{eqnarray}
By iteration,
\begin{eqnarray}
\|\nabla p\|_{q, B_{\frac{r}{2}}(x_0)}&\leq& (cK_1)^n\|\nabla p\|_{q, B_{ r_{n}}(x_0)}+\frac{c}{r}K_2\sum_{i=0}^{n-1}(2cK_1)^i+cK_3\sum_{i=0}^{n-1}(cK_1)^i.\label{iter}
\end{eqnarray}
In view of \eqref{rep5}, we can take $ r $ so that
\begin{equation}\label{conko}
	c2K_1\leq \frac{1}{2}.
\end{equation}
Then let $n\ra\infty$ in \eqref{iter} to get
\begin{eqnarray}
	\|\nabla p\|_{q, B_{\frac{r}{2}}(x_0)}&\leq&\frac{c}{r}K_2+cK_3\nonumber\\
	&=&cr^{-\frac{N}{\ell}}(1+\|\w\|_{\infty, B_{ r}(x_0)})^{5N-1}\|\nabla\m\|_{q, B_{ r}(x_0)}\|p\|_{q, B_{ r}(x_0)}\nonumber\\
	&&+\frac{c}{r}\left(1+\|\w\|_{\infty, B_{ r}(x_0)}\right)^{5N}\|p\|_{q, B_{ r}(x_0)}+\frac{c}{r}\left(1+\|\w\|_{\infty, B_{ r}(x_0)}\right)^{3N}\|\nabla p\|_{\frac{Nq}{N+q}, B_r(x_0)}\nonumber\\
&&+c\left(1+\|\w\|_{\infty, B_{ r}(x_0)}\right)^{3N-2}\|S\|_{\frac{Nq}{N+q}, B_r(x_0)}.\label{bl1}
\end{eqnarray}
By virtue of \eqref{rep5}, for \eqref{conko} to hold, it is enough for us to take
\begin{equation}\label{rdef}
	cr^{1-\frac{N}{\ell}}\left(1+\|\w\|_{W^{1,\ell} (\Omega)}\right)^{5N}=\frac{1}{4}.
\end{equation}
This combined with \eqref{bl1} yields
\begin{eqnarray}
	\|\nabla p\|_{q, B_{\frac{r}{2}}(x_0)}&\leq&
	c\left(1+\|\w\|_{W^{1,\ell} (\Omega)}\right)^{\frac{5N(2\ell-N)}{\ell-N}}\|p\|_{q, \Omega}+c\left(1+\|\w\|_{W^{1,\ell} (\Omega)}\right)^{\frac{N(8\ell-3N)}{\ell-N}}\|\nabla p\|_{\frac{Nq}{N+q}, \Omega}\nonumber\\
	&&+c\left(1+\|\w\|_{\infty, \Omega}\right)^{3N-2}\|S\|_{\frac{Nq}{N+q},\Omega}\nonumber\\
	&\leq&	c\left(1+\|\w\|_{W^{1,\ell} (\Omega)}\right)^{\frac{5N(2\ell-N)}{\ell-N}}\left(\|\nabla p\|_{\frac{Nq}{N+q}, \Omega}+\|S\|_{\frac{Nq}{N+q},\Omega}\right).
\end{eqnarray}
Here we have applied
  Poincar\'{e}'s inequality to $p$ and the Sobolev embedding theorem to $\w$.
  
If $x_0\in\partial\Omega$, the same estimate still holds with $B_{ r}(x_0)$ (resp. $B_{\frac{r}{2}}(x_0)$) being replaced by $B_{ r}(x_0)\cap\Omega$ (resp. $B_{ \frac{r}{2}}(x_0)\cap\Omega$). This can be achieved by the classical technique of flattening $B_{r}(x_0)\cap\partial\Omega$ and then turning $x_0$ into an interior point (\cite{GT}, p.300). Also see \cite{X4} for a rather detailed implementation of the technique.
We shall omit it here.

Finally, let $r$ be determined by \eqref{rdef}. There is an integer $j$ with the property
\begin{eqnarray}
	j-1\leq\frac{\textup{diam}(\Omega)}{r}\leq j.\label{rep6}
\end{eqnarray}
We can find at most $(2j)^N$ balls $\{B_{ r}(x_0^{(i)})\}$ with
$$\Omega\subset\cup_{i=1}^{(2j)^N}B_{ \frac{r}{2}}(x_0^{(i)}).$$
Consequently,
\begin{eqnarray}
		\|\nabla p\|_{q, \Omega}&\leq&\sum_{i=1}^{(2j)^N}	\|\nabla p\|_{q, B_{ \frac{r}{2}}(x_0^{(i)})}\nonumber\\
			&\leq&cj^N\left(1+\|\w\|_{W^{1,\ell} (\Omega)}\right)^{\frac{5N(2\ell-N)}{\ell-N}}\left(\|\nabla p\|_{\frac{Nq}{N+q}, \Omega}+\|S\|_{\frac{Nq}{N+q},\Omega}\right).
	\label{rep7}
\end{eqnarray}
Observe from \eqref{rep6} and \eqref{rdef} that
\begin{eqnarray}
	j^N&\leq& \frac{c}{r^N}+c\nonumber\\
	&\leq& c\left(1+\|\nabla\w\|_{W^{1,\ell} (\Omega)}\right)^{\frac{5N^2\ell}{\ell-N}}+c.\nonumber
\end{eqnarray}
Substitute this into \eqref{rep7}
to obtain
\begin{eqnarray}
	\|\nabla p\|_{q, \Omega}&\leq& c(1+\|\nabla\w\|_{W^{1,\ell} (\Omega)})^{\frac{5N(2\ell-N+N\ell)}{\ell-N}}\left(\|\nabla p\|_{\frac{Nq}{N+q},\Omega}+\|S\|_{\frac{Nq}{N+q},\Omega}\right).\label{rep8}
\end{eqnarray}
On account of \eqref{conq} and the interpolation inequality (\cite{GT}, p.146), we have
$$\|\nabla p\|_{\frac{Nq}{N+q}, \Omega}\leq \ve\|\nabla p\|_{q,\Omega}+\frac{1}{\ve^{\frac{(N-1)q-N}{q}}}\|\nabla p\|_{1,\Omega},\ \ \ve>0.$$
Plug this into \eqref{rep8} and choose $\ve$ appropriately in the resulting inequality to get
\begin{eqnarray}
	\|\nabla p\|_{q, \Omega}&\leq& c\left(1+\|\nabla\w\|_{W^{1,\ell} (\Omega)}\right)^{\frac{5N(2\ell-N+N\ell)(Nq-N)}{q(\ell-N)}}\left(\|\nabla p\|_{1,\Omega}+\|S\|_{\frac{Nq}{N+q},\Omega}\right).\label{rep9}
\end{eqnarray}
 The proof is complete.
\end{proof}

Note that if we wish to further weaken $\w$ to a VMO function as did in \cite{CFL,CFL1,DI} we will run into a technical problem. That is,  we do not know how the constant $c$ in inequality (2.3) of \cite{CFL} depends on $\|k\|_{2,\partial B_1(0)}$. 

\section{Blow-up time}
In this section we offer the proof of the main theorem. 

Assume that $(\m, p)$ is a strong solution to \eqref{e1}-\eqref{e4}. The existence of such an ``approximate'' solution will be made clear later.

By virtue of the boundary condition \eqref{e3}, we have
$$\|\m\|_{W^{1,\ell} (\Omega)}\leq c\|\nabla\m\|_{\infty,\Omega}.$$
Use $p$ as a test function in \eqref{peqn1} to get
$$\io|\nabla p|^2dx\leq \io S(x)p dx\leq \|S\|_{2,\Omega}\|p\|_{2,\Omega}\leq c\|S\|_{2,\Omega}\|\nabla p\|_{2,\Omega}.$$
Thus,
$$\|\nabla p\|_{2,\Omega}\leq c\|S\|_{2,\Omega}.$$
Thus we can write \eqref{npb10} as
\begin{equation}\label{hope1}
		\|\nabla p\|_{4q,\Omega}\leq c\left(1+\|\nabla\m\|_{\infty,\Omega}\right)^{s_1}\|S\|_{\frac{4Nq}{N+4q},\Omega}.
	\end{equation}
Here we have replaced $q$ by $4q$.
Take the $(4q)^{\mbox{th}}$ power, integrate over $(0,T)$, and then take the $(4q)^{\mbox{th}}$ root to derive
\begin{equation}\label{hope2}
	\|\nabla p\|_{4q,\ot}\leq cT^{\frac{1}{4q}}\left(1+\|\nabla\m\|_{\infty,\Omega}\right)^{s_1}\|S\|_{\frac{4Nq}{N+4q},\Omega}.
\end{equation}
The rest of the proof of Theorem \ref{main} is divided into several lemmas.
\begin{lemma}
	We have
	\begin{eqnarray}
		\sup_{\ot}|\m|&\leq &	cT^{\frac{s_0}{4q}}\left(T^{\frac{1}{2}}\|\m_0\|_{2,\Omega}+T\|S\|_{2,\Omega}\right)\|\nabla p\|_{4q,\ot}^{s_0}\nonumber\\
		&&+c\left(\|\m_0\|_{\infty,\Omega}+T^{\frac{1}{2}}\|\m_0\|_{2,\Omega}+T\|S\|_{2,\Omega}\right),\label{hope6}
	\end{eqnarray}
	where $s_0$ is given as \eqref{sdef}.
\end{lemma}
\begin{proof}
Take the dot product of \eqref{e2} with $\m$ to derive
\begin{eqnarray}	
	\frac{1}{2}\partial_t|\m|^2-\frac{D^2}{2}\Delta|\m|^2+D^2|\nabla\m|^2+|\m|^{2\gamma}=E^2\mnp^2\leq E^2|\nabla p|^2|\m|^2\ \ \mbox{in $\ot$.}\label{ms1}
\end{eqnarray}
Drop the two non-negative terms on the left-hand side and then  apply \eqref{u7} to derive
\begin{eqnarray}
	\sup_{\ot}|\m|^2&\leq &	c\sup_{\Omega}|\m_0|^2+c\left(\|\nabla p\|_{2q,\ot}^{2s_0}+1\right)\|\m\|_{2,\ot}^2.\label{hope5}
\end{eqnarray}
We can deduce from \eqref{me1} that 
$$\sup_{0\leq t\leq T}\io|\m(x,t)|^2dx\leq c\io|\m_0(x)|^2dx+cT\|S(x)\|^2_{2,\Omega}.$$
Use this in \eqref{hope5} to get \eqref{hope6}.
\end{proof}
\begin{lemma}
	We have
	\begin{eqnarray}
		\|\nabla\m\|_{\infty,\ot}\leq cG(T) +cF(T)\|\nabla p\|_{4q,\ot}^{s_4},\label{nmub4}
	\end{eqnarray}
where $s_4$ is a positive number determined by $N, q, \gamma$,
\begin{eqnarray}
	F(T)&=&T^{\frac{s_0}{4q}}\left(T^{\frac{1}{2}}\|\m_0\|_{2,\Omega}+T\|S(x)\|_{2,\Omega}\right)\nonumber\\
	&&+\left(\|\m_0\|_{\infty,\Omega}+T^{\frac{1}{2}}\|\m_0\|_{2,\Omega}+T\|S(x)\|_{2,\Omega}\right)\nonumber\\
	&&+T^{\frac{s_0(2\gamma-1)+2}{4q}}\left(T^{\frac{2\gamma-1}{2}}\|\m_0\|^{2\gamma-1}_{2,\Omega}+T^{2\gamma-1}\|S(x)\|^{2\gamma-1}_{2,\Omega}\right),\ \ \mbox{and}\nonumber\\
	G(T)&=&\|\nabla\m_0\|_{\infty,\Omega}+F(T)+T^{\frac{1}{2q}}\left(\|\m_0\|^{2\gamma-1}_{\infty,\Omega}+T^{\frac{2\gamma-1}{2}}\|\m_0\|^{2\gamma-1}_{2,\Omega}+T^{2\gamma-1}\|S(x)\|^{2\gamma-1}_{2,\Omega}\right).
\end{eqnarray}
\end{lemma}
\begin{proof}
We can write \eqref{e2} in the form
\begin{equation}
	\pt m_i-D^2\Delta m_i=E^2\mnp p_{x_i}-|\m|^{2(\gamma-1)}m_i\ \ \mbox{in $\ot$,}\ \ i=1,\dots,N.
\end{equation}
This puts us in a position to apply Lemma \ref{nuub}. Upon doing so, we arrive at
\begin{equation}\label{nmub}
	\| \nabla m_i\|_{\infty,\ot}\leq c\|\nabla\m_0\|_{\infty,\Omega}+c\left\|E^2\mnp p_{x_i}-|\m|^{2(\gamma-1)}m_i\right\|_{2q,\ot}.
\end{equation}
We estimate from \eqref{hope6} that
\begin{eqnarray}
\lefteqn{	\left\|E^2\mnp p_{x_i}-|\m|^{2(\gamma-1)}m_i\right\|_{2q,\ot}}\nonumber\\
&\leq &c\|\m\|_{\infty,\ot}\|\nabla p\|_{4q,\ot}^2+cT^{\frac{1}{2q}}\|\m\|_{\infty,\ot}^{2\gamma-1}\nonumber\\
&\leq &	cT^{\frac{s_0}{4q}}\left(T^{\frac{1}{2}}\|\m_0\|_{2,\Omega}+T\|S(x)\|_{2,\Omega}\right)\|\nabla p\|_{4q,\ot}^{s_0+2}\nonumber\\
	&&+c\left(\|\m_0\|_{\infty,\Omega}+T^{\frac{1}{2}}\|\m_0\|_{2,\Omega}+T\|S(x)\|_{2,\Omega}\right)\|\nabla p\|_{4q,\ot}^2\nonumber\\
&&+	cT^{\frac{s_0(2\gamma-1)+2}{4q}}\left(T^{\frac{2\gamma-1}{2}}\|\m_0\|^{2\gamma-1}_{2,\Omega}+T^{2\gamma-1}\|S(x)\|^{2\gamma-1}_{2,\Omega}\right)\|\nabla p\|_{4q,\ot}^{s_0(2\gamma-1)}\nonumber\\
	&&+cT^{\frac{1}{2q}}\left(\|\m_0\|^{2\gamma-1}_{\infty,\Omega}+T^{\frac{2\gamma-1}{2}}\|\m_0\|^{2\gamma-1}_{2,\Omega}+T^{2\gamma-1}\|S(x)\|^{2\gamma-1}_{2,\Omega}\right).\label{rep10}
\end{eqnarray}
Set
\begin{equation}
	s_4=\max\{s_0+2, s_0(2\gamma-1)\}.
\end{equation}
Then Young's inequality asserts
\begin{eqnarray}
	\|\nabla p\|_{4q,\ot}^{s_0+2}&\leq &c\|\nabla p\|_{4q,\ot}^{s_4}+c,\\
			\|\nabla p\|_{4q,\ot}^{2}&\leq &c\|\nabla p\|_{4q,\ot}^{s_4}+c,\\
				\|\nabla p\|_{4q,\ot}^{s_0(2\gamma-1)}&\leq &c\|\nabla p\|_{4q,\ot}^{s_4}+c.
\end{eqnarray}
Combining this with \eqref{rep10} and \eqref{nmub} yields \eqref{nmub4}.
\end{proof}
Use \eqref{nmub4} in \eqref{hope2} to obtain
\begin{eqnarray}
	\|\nabla p\|_{4q,\ot}&\leq& cG_1(T) +cF_1(T)\|\nabla p\|_{4q,\ot}^{s_5},\label{re1}
\end{eqnarray} 
where
\begin{equation}
	G_1(T)=T^{\frac{1}{4q}}\left(G^{s_1}(T)+1\right)\|S\|_{\frac{4Nq}{N+4q},\Omega},\ \ F_1(T)=T^{\frac{1}{4q}}F^{s_1}(T)\|S\|_{\frac{4Nq}{N+4q},\Omega},\ \ s_5=s_1s_4.
\end{equation}
Note that we can represent $ G_1(T)$ as the sum of $12$ terms, each of which is of the form $$T^a\|S\|_{\frac{4Nq}{N+4q},\Omega}^b\|\nabla\m_0\|_{\infty,\Omega}^c$$ with $a>0, b\geq 0, c\geq 0$ being determined by $N, q, \gamma$ only. The same can be done for $F_1(T)$ except that there are only seven terms in $F_1(T)$. 
It  follows from \eqref{re1} that
\begin{equation}
\|\nabla p\|_{4q,\ott}\leq cF_1(T)\|\nabla p\|_{4q,\ott}^{s_5}+ c G_1(T)\ \ \mbox{for each $\tau\in [0,T]$.}\label{npb1}
\end{equation}
Now we are in the situation of Lemma \ref{prop2.2} where $h(\tau)=\|\nabla p\|_{4q, \ott}$. Following the proof of the lemma in \cite{X4}, we consider the function $g(h)=cF_1(T)h^{s_5}-h+cG_1(T)$ on $[0,\infty)$. Then \eqref{npb1} implies
\begin{equation}\label{npb2}
	g\left(h(\tau)\right)\geq 0 \ \ \mbox{for each $\tau\in [0,T]$.}
\end{equation}
We compute
$$g^\prime(h)=cs_5F_1(T)h^{s_5-1}-1.$$
Thus, $g(h)$ is decreasing on $\left(0, \frac{1}{(cs_5F_1(T))^{\frac{1}{s_5-1}}}\right)$ and increasing on $\left(\frac{1}{(cs_5F_1(T))^{\frac{1}{s_5-1}}}, \infty\right)$. The minimum value of $g$ is given by
\begin{equation}\label{gmin}
	m_g\equiv	g\left(\frac{1}{(cs_5F_1(T))^{\frac{1}{s_5-1}}}\right)=-\frac{s_5-1}{s_5(cs_5F_1(T))^{\frac{1}{s_5-1}}}+cG_1(T).
\end{equation}
It is easy to see that there is a unique solution  $T_{\textup{max}}$ to the equation
\begin{equation}\label{tpre}
	\frac{s_5-1}{s_5(cs_5F_1(T_{\textup{max}}))^{\frac{1}{s_5-1}}cG_1(T_{\textup{max}})}=1.
\end{equation}
We have 
$$	m_g<0\ \ \mbox{for each $T<T_{\textup{max}}$.}$$
Fix  $T<T_{\textup{max}}$. Observe that $\|\nabla p\|_{4q,\ott}$ is a continuous function of $\tau$ and $\lim_{\tau\rightarrow 0}\|\nabla p\|_{4q,\ott}=0$. We can infer from \eqref{npb2} that
\begin{equation}
	\|\nabla p\|_{4q,\ott}<\frac{1}{(cs_5F_1(T))^{\frac{1}{s_5-1}}}\ \ \mbox{for each $\tau\leq T$.}
\end{equation}
In particular,
\begin{equation}
	\|\nabla p\|_{4q,\ot}<\frac{1}{(cs_5F_1(T))^{\frac{1}{s_5-1}}}.
\end{equation}
Then (D4) follows from  the classical regularity result for linear parabolic equations in (\cite{LSU}, p.204).
We can easily transform \eqref{tpre} into \eqref{tmax}.


The existence of a solution in the preceding calculations can be established via the Leray-Schauder fixed point theorem (\cite{GT}, p.280). To this end, let $q$ be given as in (H1)
and set
$$\mathcal{B}=L^{4q}\left(0,T;W_0^{1,4q}(\Omega)\right)
.$$
Then define an operator $\mathcal{T}$ from $\mathcal{B}$ into itself as follows: Let $p\in \mathcal{B} $. We say $w=\mathbb{T}(p)$ if $w$ is the unique solution of the problem 
\begin{eqnarray}
-\mdiv\left((I+\n\otimes\n)\nabla w\right)&=&S(x)\ \ \mbox{in $\ot$,}\\
w&=&0\ \ \mbox{on $\Sigma_T$},
\end{eqnarray}
where $\n$ solves the problem
\begin{eqnarray}
	\pt \n-D^2\Delta \n+|\n|^{2(\gamma-1)}\n&=&E^2\nnp\nabla p\ \ \mbox{in $\ot$,}\label{eqn1}\\
\n&=&0\ \ \mbox{on $\Sigma_T$,}\label{eqn2}\\
	\n(x,0)&=&\m_0(x)\ \ \mbox{on $\Omega$.}\label{eqn3}
\end{eqnarray}
We have:
\begin{lemma}\label{exis} Let (H1)-(H3) hold.
Then there is a unique weak solution $\n$ to \eqref{eqn1}-\eqref{eqn3} in the space $C\left([0,T];\left(L^2(\Omega)\right)^N\right)\cap L^2\left((0,T);\left(W_0^{1,2}(\Omega)\right)^N\right)$. Furthermore, 
	\begin{equation}\label{nhc}
	\mbox{$\n$ is H\"{o}lder continuos in $\overline{\Omega_T}$ with $\n\in L^\infty(0,T; W^{1,\infty}_0(\Omega))$.}	
	\end{equation}
\end{lemma}
We postpone the proof of this lemma to the end of the section. Equipped with this lemma, we can claim that $\mathcal{T}$ is well-defined. Under \eqref{nhc} and (H3), we can appeal to \eqref{r82}, thereby yielding
\begin{equation}
	\|\nabla w\|_{4q,\Omega}\leq c\|S(x)\|_{\frac{4Nq}{N+4q},\Omega},
\end{equation}
from whence follows
\begin{equation}
	\|\nabla w\|_{4q,\Omega_T}\leq cT^{\frac{1}{4q}}\|S(x)\|_{\frac{4Nq}{N+4q},\Omega}.
\end{equation}
For $\mathcal{T}$ to have a fixed point, we must verify
\begin{enumerate}
	\item[(C1)]$\mathcal{T}$ is continuous;
	\item[(C2)]$\mathcal{T}$ maps bounded sets into precompact ones;
	\item[(C3)] There is a constant $c$ such that 
	\begin{equation}
		\|p\|_\mathcal{B}\leq c
	\end{equation}
	for all $p\in \mathcal{B}$ and $\sigma\in [0,1]$ satisfying
	\begin{equation}\label{t1}
		p=\sigma\mathcal{T}(p).
	\end{equation}
\end{enumerate}
To see (C2), suppose
$$p_n\ra p\ \ \mbox{weekly in $L^{4q}\left(0,T;W_0^{1,4q}(\Omega)\right)$.}$$
Denote by $\n_n$ the solution to \eqref{eqn1}-\eqref{eqn3} with $p$ being replaced by $p_n$. In view of \eqref{nhc}, we can extract a sub-sequence of $\{\n_n\}$, not relabeled, such that
\begin{equation}
	\n_n\ra \n\ \ \mbox{uniformly in $\left(C(\overline{\Omega_T})\right)^N$.}
\end{equation}
Note that we have
\begin{equation}\label{pne}
	-\mdiv\left[(I+\n_n\otimes\n_n)\nabla p_n\right]= S(x)\ \ \mbox{in $\ot$.}
\end{equation}
Thus, we can pass to the limit in the above equation to get
\begin{equation}
	-\mdiv\left[(I+\n\otimes\n)\nabla p\right]= S(x)\ \ \mbox{in $\ot$.}
\end{equation}
Subtract this equation fro \eqref{pne} to derive
$$	-\mdiv\left[(I+\n_n\otimes\n_n)\nabla (p_n-p)\right]=\mdiv\left[(\n_n\otimes\n_n-\n\otimes\n)\nabla p\right] \ \ \mbox{in $\ot$.}  $$
Once again,  we can use \eqref{r82} to get
$$\|\nabla (p_n-p)\|_{4q,\ot}\leq c\|(\n_n\otimes\n_n-\n\otimes\n)\nabla p\|_{4q,\ot}\leq c\|\n_n\otimes\n_n-\n\otimes\n\|_{\infty,\ot}\ra 0.$$
That is, (C2) holds. Each problem in the definition of $\mathcal{T}$ has a unique solution. This together with (C2) implies (C1).

We easily see that \eqref{t1} is equivalent to the problem
\begin{eqnarray}
	-\mdiv\left((I+\m\otimes\m)\nabla p\right)&=&\sigma S(x)\ \ \mbox{in $\ot$,}\\
	\pt \m-D^2\Delta \m+|\m|^{2(\gamma-1)}\m&=&E^2\mnp\nabla p\ \ \mbox{in $\ot$,}\\
	p&=&0\ \ \mbox{on $\Sigma_T$,}\\
	\m&=&0\ \ \mbox{on $\Sigma_T$,}\\
	\m(x,0)&=&\m_0(x)\ \ \mbox{on $\Omega$.}
\end{eqnarray}
Now we are in a position to employ the earlier argument. Upon doing so, we obtain (C3) for each $T<T_{\textup{max}}$.
\begin{proof}[Proof of Lemma \ref{exis}]We first establish the uniqueness assertion. Suppose that \eqref{eqn1}-\eqref{eqn3} has two solutions, say, $\n_1, \n_2$. Then $\n\equiv\n_1-\n_2$ satisfies 
	\begin{eqnarray}
		\pt \n-D^2\Delta \n+|\n_1|^{2(\gamma-1)}\n_1-|\n_2|^{2(\gamma-1)}\n_2&=&E^2(\n\cdot\nabla p)\nabla p\ \ \mbox{in $\ot$,}\label{eqn4}\\
		\n&=&0\ \ \mbox{on $\Sigma_T$,}\label{eqn5}\\
		\n(x,0)&=&0\ \ \mbox{on $\Omega$.}\label{eqn6}
	\end{eqnarray}
Recall from Lemma \ref{plap} that	
	$$\left(|\n_1|^{2(\gamma-1)}\n_1-|\n_2|^{2(\gamma-1)}\n_2\right)\cdot\n\geq 0.$$
With this in mind, we take the dot product of \eqref{eqn4} with $\n$ to derive
\begin{eqnarray}
	\frac{1}{2}\sup_{0\leq t\leq T}\io|\n|^2dx+D^2\ioT|\nabla\n|^2dxdt&\leq&2 E^2\ioT\nnp^2dxdt.
\end{eqnarray}	
By Poincar\'{e}'s inequality,
\begin{eqnarray}
	\ioT|\n|^{2+\frac{4}{N}}dxdt&\leq &\int_{0}^{T}\left(\io|\n|^{\frac{2N}{N-2}}dx\right)^{\frac{N-2}{N}}\left(\io|\n|^2dx\right)^{\frac{2}{N}}dt\nonumber\\
		&\leq&c\sup_{0\leq t\leq T}\left(\io|\n|^2dx\right)^{\frac{2}{N}}\ioT|\nabla\n|^2dxdt\nonumber\\
		&\leq &c\left(\ioT\nnp^2dxdt\right)^{\frac{N+2}{N}}\nonumber\\
		&\leq&c\ioT|\n|^{2+\frac{4}{N}}dxdt\left(\ioT|\nabla p|^{N+2}dxdt\right)^{\frac{2}{N}}.\label{eqn7}
\end{eqnarray}
Here $c$ depends only on $N, D, E$. Obviously, we can pick a positive number $\tau\leq T$ such that
$$c\left(\iot|\nabla p|^{N+2}dxdt\right)^{\frac{2}{N}}<1.$$
Then \eqref{eqn7} implies
$$\n=0\ \ \mbox{in $\Omega\times[0, \tau]$.}$$
If $\tau<T$, then we apply the preceding proof to the problem on $\Omega\times(\tau, T)$. In a finite number of steps, we can achieve
$$\n=0\ \ \mbox{in $\ot$.}$$

To establish \eqref{nhc}, it is enough for us to show that 
\begin{eqnarray}
	\n\in\left(L^\infty(\ot)\right)^N.\label{nubb}
\end{eqnarray} To see this, we write \eqref{eqn1} as
\begin{equation}
	\pt \n-D^2\Delta \n=E^2\nnp\nabla p-|\n|^{2(\gamma-1)}\n\in L^q(\ot).
\end{equation}
Since $q>1+\frac{N}{2}$, we can invoke the classical result in (\cite{LSU}, p.204) and Lemma \ref{nmub1} to conclude \eqref{nhc}.
As for \eqref{nubb}, we take the dot product of \eqref{eqn1} with $\n$ to deduce
$$\pt|\n|^2-D^2\Delta|\n|^2\leq 2E^2|\nabla p|^2|\n|^2\ \ \mbox{in $\ot$}.$$
We can infer \eqref{nubb} from Lemma \ref{gub}.

The existence of a weak solution to \eqref{eqn1}-\eqref{eqn3} can also be established via the Leray-Schauder fixed point theorem. In this case, we define an operator $B$ from $\left(L^\infty(\ot)\right)^N$ into itself as follows: For each $\m\in \left(L^\infty(\ot)\right)^N$ we let $\w=B(\m)$ be the unique solution of the problem
\begin{eqnarray}
	\pt\w-D^2\Delta\w&=&E^2\mnp\nabla p-|\m|^{2(\gamma-1)}\m\ \ \mbox{in $\ot$,}\label{we1}\\
	\w&=&0\ \ \mbox{on $\Sigma_T$,}\nonumber\\
	\w&=&\m_0(x)\ \ \mbox{on $\Omega$.}\nonumber
\end{eqnarray}
The term on the right-hand side of \eqref{we1} lies in $L^{2q}(\ot)$. Thus $\w$ is H\"{o}lder continuous on $\overline{\ot}$.
With this in mind, we can easily verify that (C1)-(C3) are all satisfied by $B$. This completes the proof.
\end{proof}

\end{document}